\documentclass[11pt]{article}
\usepackage{amsmath,amssymb,amsthm,color,graphicx}
\usepackage[latin1]{inputenc}
\usepackage{amsfonts}
\usepackage{wasysym}
\newcommand{\R}{\mathbb R}

\newcommand{\vm}{\vec{m}}

\newcommand{\vp}{\vec{p}}

\newcommand{\vzeta}{\vec{\zeta}}

\newcommand{\vphi}{\vec{\phi}}
\newcommand{\vmu}{\vec{\mu}}
\newcommand{\vnu}{\vec{\nu}}

\newcommand{\vpsi}{\vec{\psi}}
\newcommand{\vA}{\vec{\cal A}}

\newcommand{\cM}{\mathbb{M}}

\newtheorem{theorem}{Theorem}

\newtheorem{proposition}{Proposition}

\newtheorem{defi}{Definition}
\newtheorem{acknowledgment*}{Acknowledgment}

\newtheorem{prob}{Problem}
\newcommand{\be}{\begin{equation}}
\newcommand{\ee}{\end{equation}}

\begin{document}
		\title{On semi-discrete sub-partitions of vector-valued measures}
	\author{Shlomi Gover, Gershon Wolansky \\ Department of Mathematics, Technion, Haifa 32000, Israel}
	
	\maketitle
	\begin{center} Abstract \end{center}
	We introduce a concept of optimal transport for vector-valued measures and its dual formulation. In this note we concentrate  on the semi-discrete case and show some fundamental differences between the scalar and vector cases. A manifestation of this difference is the possibility of  non-existence of optimal solution for the dual problem for  feasible primer problems. 
\section{Introduction}
There are $n$  agents sharing a  cake of $q$ layers (possibly of different widths). Each agent demands a given amount of each of the components. Assume the demands are all feasible (i.e. the sum of the demands for each component by all agents does not exceed the  total amount of this component).  Can we split the cake {\em vertically} (i.e.without separating the layers) such that each agent will get his precise demand? If such a sub-partition is possible, in how many ways can it be done?

Assume further that the cost of production and delivery of each part of the cake depends on the consuming agent
(say, some agents demand a better quality and/or the the cost  of delivery for some agents is higher than for others). Can the market determine an equilibrium price that the baker may charge  for each component for each agent, in a way that will supply  the agents' demand?

In the semi-discrete setting we adopt, the number of agents and layers are finite but the cake itself  is a continuum.
If the number of components (layers) of the cake is one then this is a special case of optimal transport (Monge-Kantorovich)  theory (see, e.g. \cite{SntA, Vil1, Vil2}).  In optimal transport both the cake and    agents can be  continuum, but there is only a single layer. In the single-layer semi discrete setting  the first question is trivial  since any feasible demand can always be satisfied, and, in general, there are   infinity many ways to do it.  It is less trivial but still true that  an equilibrium  price can also be determined for each agent in that case.

The problem is much less trivial in the multi-layer case. In this note we  characterize the demands that can be supplied and the pricing strategy in this case, and  touch on some other issues related to multi-layer optimal transport

\section{The single layer}
 
 Let $(X,d,\mu)$ be a compact metric measure space (the "cake") and ${\cal B}$ the Borel sigma-algebra over $X$. The Radon measure $\mu:{\cal B}\rightarrow \R_+$ is assumed to  contain no atoms, in particular $0<\mu(X)<\infty$ and for any $m\in(0, \mu(X))$ there are sets $A, A^{'}\in{\cal B}$  (in fact, infinitely many such sets) for which 
 $\mu(A)=\mu(A^{'})=m$ and $\mu(A\Delta A^{'})>0$  ($\Delta$ stands for the symmetric difference). A {\em sub-partition} of $X$ corresponding to $\vm=(m_1,\ldots m_n)\in\R^n_+$ (the "demands")  is composed of $n$ essential disjoint sets $A_1, \ldots A_n\in {\cal B}$ such that 
 $$ \mu(A_i)=m_i \ \ \text{and} \ \ \mu(A_i\cap A_k)=0\ \ \text{for} \ i\not= k \ . $$
  We identify two sub-partitions $\vec{A}:=(A_1\ldots A_n)$, $\vec{A}^{'}:=(A^{'}_1\ldots A^{'}_n)$ if $\mu(A_i\Delta A^{'}_i)=0$ for $i=1\ldots n$.
 Determine the set of all these sub-partitions by $\vA(\mu,\vm)$. 
 Evidently, there exists infinity many   sub-partitions in $\vA(\mu,\vm)$ if $|\vm|:= \sum_1^n m_i \leq \mu(X)$ and non if $|\vm|>\mu(X)$.
 
 Let $\vec{c}:= (c_1\ldots c_n):X\rightarrow \R^n_+$ be a continuous cost function. A {\em stable sub-partition} is a sub-partition $\vec{A}:=(A_1, \ldots A_n)\in \vA(\mu,\vm)$ 
 which minimizes 
 \be\label{sd} Mo(\vec{A}):=\sum_{i=1}^n\int_{A_i}c_id\mu\ee
 over all $\vec{A}\in \vA(\mu,\vm)$. 
 
 The components  $c_i(x)$  of $\vec{c}$ at $x\in X$ are the "cost of production" of $x$ for $i$.  Let $p_i\in\R$ be the price of   purchase per unit of agent $i$ \cite{Gal}. The net profit of the "baker" 
for selling $x$ to agent $i$ is $p_i-c_i(x)$. It is assumed that the baker will not sell $x$ to $i$ unless both conditions hold:  The profit is positive and he cannot increase it  by selling it to another agent. Thus, a necessary condition for $x$ to be sold to agent $i$ is 
\be\label{phidef}p_i-c_i(x)= \max\{ 0, p_k-c_k(x); k=\{1\ldots n\}\}:= -\phi(x, \vp) \ . \ee
 The price vector $\vp=(p_1\ldots p_n)$ is called an equilibrium price for the sub-partition $\vec{A}=(A_1, \ldots A_n)$ if 
 $$A_i\subseteq \{x; p_i-c_i(x)=-\phi(x,\vp)\}$$
 for any $i\in\{1\ldots n\}$.  It can easily be verified that 
 	a sub-partition $\vec{A}\in\vA(\mu, \vm)$ is stable for given cost if  there exists an equilibrium price vector for this sub-partition. 

 The semi-discrete case is a special case of the celebrated {\em optimal transport problem}. In general, the problem of optimal transport (Monge-Kantorovich) refers to a pair of measure  spaces  $(X,\mu), (Y,\nu)$ where 
 $\mu(X)=\nu(Y)$, and $c:X\times Y\rightarrow \R$ is a given cost function. The Monge problem \cite{Mo} is the minimization of the functional
 $$ Mo(T)=\int_Xc(x, T(x))\mu(dx)$$
 over all measurable $T:X\rightarrow Y$ which satisfies the condition $T_\#\mu=\nu$, namely
 $\mu(T^{-1}(B))=\nu(B)$ for any measurable set $B\subset Y$. 
 The Kantorovich relaxation reduces the Monge problem to infinite-dimensional linear programming:
 \paragraph{Kantorovich problem:} \cite{Ka} Minimize 
\be\label{Kanto} K(\pi):=\int_X\int_Y c(x,y)\pi(dxdy)\ee
 over  $\pi\in\Pi(\mu,\nu)$ 
 where  
 \be\label{Pidef} \Pi(\mu,\nu):=\left\{ \pi\  \text{is a positive measure on} \ \ X\times Y, \ \ 
 \int_X\pi(dxdy)= \nu(dy), \ \ \int_Y \pi(dxdy)=\mu(dx) \right\} \ee
 Under mild  conditions  the minimizer of the Kantorovich problem always exists, and \cite{BK, Partelli} $$\min_{\pi\in\Pi(\mu,\nu)} K(\pi)=\inf_{T_\#\mu=\nu}Mo(T)\ . $$
 
 The problem of equilibrium prices is related to the Kantorovich duality Theorem. 
 
 \par\noindent{\bf Kantorovich Duality Theorem}\cite{Rec1}: {\it 
 \be\label{duali}\min_{\pi\in\Pi(\mu,\nu)} K(\pi)=\sup_{(\phi,\psi)\in J(c)}\int_X\phi d\mu+\int_Y\psi d\nu\ee
 where 	
 \be\label{duali1} J(c):= \left\{ \phi\in C(X), \psi\in C(Y); \psi(x)+\phi(y)\leq c(x,y)\ \ \forall (x,y)\in X\times Y  \right\} \ . \ee
 }

 In the semi discrete case (\ref{sd}) $Y$ is a finite set ($Y=\{0,1, \ldots n\}$), 
 $\nu(\{i\})=m_i$, $i=1\ldots n$, $\nu(\{0\})=\mu(X)-|\vm|$ and $c(x,i)=c_i(x)$, $c_0(x)\equiv 0$. An optimal sub-partition corresponds to  the solution of the Monge problem $T$ via $A_i=T^{-1}(\{i\})$ where $i\in \{1\ldots n\}$. 
 \par The existence of such an optimal mapping $T$ is not evident in the general case, but there always exists an optimal sub-partition in the semi-discrete case \cite{wol1}.  The equilibrium price for the semi-discrete case corresponds to the optimal function $\psi$ in the dual problem, which, in this case ($Y=\{0,1\ldots n\}$), is just a "price vector"  $\vp=(p_1, \ldots p_n)\in\R^n$ where $p_0$ (price of the "null agent") is set to zero.  Thus $\int_Y\psi d\nu\equiv \vp\cdot\vm$.  The optimal function  $\phi$ maximizing the right side of (\ref{duali}) for a given $\psi= \vp$ subject to the constraint (\ref{duali1}) is given by (\ref{phidef}).  In the semi-discrete case the Kantorovich duality Theorem takes the form
 \be\label{dualsd} \inf_{\vec{A}\in\vA(\mu,\vm)} Mo(\vec{A})=\sup_{\vp\in \R^n} \vp\cdot\vm - \int_X\phi(x,\vp) \mu(dx) \ . \ee
 
 {\bf Theorem}
 \cite{wol,wol1}: 
 {\it 	If $|\vm|\leq \mu(X)$ and $\mu$ is non-atomic then there exists a stable sub-partition $\vec{A}\in\vA(\mu,\vm)$ minimizing the left side of (\ref{dualsd}), and an equilibrium price vector $\vp$ for this sub-partition maximizing the right side of (\ref{dualsd}). If, moreover, 
 	$$ \mu(x\in X; c_i(x)-c_k(x)=\lambda)=0$$
 	for any $i\not=k\in\{1\ldots n\}$ and any $\lambda\in\R$, then there exists a {\em unique} stable sub-partition for each $\vm$ in the simplex $|\vm|\leq \mu(X)$.  }

 \section{Semi discrete optimal transport in the Multi-layer case}
 Let $\vec{\mu}:=(\mu^1\ldots 
\mu^q)$ be a vector of non-negative Radon measures on ${\cal B}$ (the layers), such that $\mu=\sum_1^q\mu^j$ is  atoms-free. 
\begin{prob}Given $n\times q$ matrix ${\cal M}=\{m_i^j\}$
such that  $m_i^j\geq 0$. Is there a sub-partition of $X$ into $n$ pairwise essentially disjoint sets $A_i\in {\cal B}$  such  that $\mu^j(A_i)=m_i^j$?
	\end{prob}

Evidently, a necessary condition is the feasibility 
 \be\label{nescond}\sum_{i=1}^n m_i^j\leq \mu^j(X)\ \ \ \forall j\in \{1\ldots q\} \ . \ee
   Let $\cM(n, \vmu)$  be the set of all  $n\times q$ matrices ${\cal M}$ for which such a sub-partition exists.
   
   Note that the case $q=1$ is trivial, as $\cM(n,\vmu)$ is just the simplex of vectors  $\vm\in\R^n_+$, $\{|\vm|\leq \mu(X)\}$  and (\ref{nescond}) is also a sufficient condition for ${\cal M}$  to be in $\cM(n,\vmu)$. 
\begin{prob}\label{probunique}
	Is there ${\cal M}\in \cM(n,\vmu)$ for which there is a unique sub-partition $A_1, \ldots A_n\in {\cal B}$ such that $\mu^j(A_i)=m_i^j$? 
\end{prob}
 Let  $\zeta_j:= d\mu^j/d\mu$ be the relative density of $\mu^j$ (in particular $\sum_1^q\zeta_j(x)=1$ $\mu$-a.s). Let  $\vec{\zeta}:= (\zeta_1, \ldots \zeta_q)$. Let $c_i(x)$ be continuous functions on $X$ (the "cost of production" of $x$ for agent  $i\in\{1\ldots n\}$). 
  Let the following "generic" assumptions
 \begin{description}
 	\item[i] $\zeta_j$ are continuous on $X$ for any $j\in\{1\ldots q\}$.
 	\item[ii] For any $\vec\lambda\in\R^q$, $\vec{\lambda}\not= 0$, $\mu(x\in X; \vec\lambda\cdot\vec{\zeta}(x)=0)=0 $. 
 	\item[iii] For any $\vec\lambda\in\R^q$ and any $i\not=k\in\{1\ldots n\}$,  $\mu(x\in X; \vec\lambda\cdot\vec{\zeta}(x)=c_i(x)-c_k(x))=0 $
 \end{description}
 Let ${\cal P}:=\{p_i^j\}$ be the price of component $j$ charged from agent $i$. The net income associated with   $x\in X$ out of agent $i$ after deducing the cost is, then $\sum_{j=1}^q p_i^j\zeta_j(x) -c_i(x)$. This  determines a sub-partition by  assigning  $x\in X$ to the agent $i$ which maximizes the net income for the "baker": 
\be\label{Ap} A_i({\cal P}):= \left\{ x\in X; \sum_{j=1}^q p_i^j\zeta_j(x) -c_i(x)= \phi(x;{\cal P}) \right\}\ee
where 
$$\phi(x;{\cal P}):= \left[\max_{k\in\{1\ldots n\}}\sum_{j=1}^q p_k^j\zeta_j(x) -c_k(x)\right]_+ \ . $$
By Assumption [{\bf iii}] above we obtain that, indeed, $\{A_i({\cal P})\}_{i=1\ldots n}$ is a sub-partition satisfying
$$ \mu\left(A_i({\cal P})\cap A_k({\cal P})\right)=0 \ \ \text{for} \ i\not= k \ \text{and} \ \cup_{i=1}^n A_i({\cal P})\subset X $$
where the last inclusion can be strict. 
\begin{defi}
The price matrix ${\cal P}$ is an equilibrium price for the demand ${\cal M}:=\{m_i^j\}\in \cM(n,\vmu)$ if the corresponding sub-partition satisfies the demand, namely 
$\mu^j(A_i({\cal P}))=m_i^j$ for each  $i\in \{1\ldots n\}$ and $j\in\{1\ldots q\}$. 
 \end{defi}
The next question we ask is
\begin{prob}\label{probprice}
Let ${\cal M}\in \cM(n,\vmu)$ and a continuous cost $\vec{c}=(c_1\ldots c_n)$ on $X$ satisfying condition {\em [{\bf iii}]}. 	Is there always an equilibrium price ${\cal P}$?
\end{prob}
It turns that the answer to Problem \ref{probprice} is related to Problem \ref{probunique}. 

In the case of a single  layer ($q=1$)  and  at least two agents  then the answer is (trivially) No to Problem \ref{probunique}  and Yes to Problem \ref{probprice}. For $q>1$ we will show that the answer to Problem \ref{probunique} is (surprisingly) Yes,  which, in turn, implies an example of non existence of equilibrium price (that is No for Problem \ref{probprice}). 
 \section{Main results}
   The first result follows from Lyapunov's theorem for vector measures \cite{Lyap} (see also \cite{Lyapel, LW})
   \par\noindent
   {\bf Theorem}[Lyapunov]:  
   	{\it The set $\{\vmu(E); \ \ E\in {\cal B}\}$ is convex in $\R^q$.}
 
An immediate characterization of $\cM(n,\vmu)$ follows from Lyapunov's theorem:
 	\begin{proposition}
 		Let $P_i$, $i=1,\ldots n$  be the projection of the $n\times q$ matrices on the $i$ row. Then $P_i\cM(n,\vmu)$ is a convex subset in $\R_+^q$.
 	\end{proposition}
 There is, however, a stronger result:
 \par\noindent 
 {\bf Theorem} \ \cite{wolar}: \ {\it 
 	The set $\cM(n,\vmu)$ is convex in the space of $n\times q$ matrices.}

 We now address  Problem \ref{probunique}. 

Let ${\cal P}:=\{\vec{p}_1, \ldots \vec{p}_n\}$, $\vec{p}_i\in\R^q$ for $i=1\ldots n$ be an $\R$-valued  $n\times q$ matrix. For each such matrix define the $n-$sub-partition $ \vA^0({\cal P})$ corresponding to (\ref{Ap}) above where $\vec{c}=0$, namely
$$ A^0_i({\cal P}):= \left\{x\in X; \vec{p}_i\cdot\vec{\zeta}=\left[\max_{k=1\ldots q} \vec{p}_k\cdot\vec{\zeta}(x)\right]_+\right\} \ \ , i=1\ldots n \ . $$
It follows from assumptions [{\bf i,ii}] that  $\mu\left(A^0_i({\cal P})\cap A^0_k({\cal P})\right)=0$ if $\vp_i\not=\vp_k$. 
In particular, $A^0_1({\cal P}), \ldots A^0_n({\cal P})$ is an essentially disjoint sub-partition if 
\be\label{ass1}\vp_i\not=\vp_k \ \text{for \  any} \  1\leq i\not= k\leq n\ . \ee
Let ${\cal M}({\cal P}):= \{m_i^j({\cal P})\}=\{\mu^j(A_i^0({\cal P}))\}$. 

	Answer to problems \ref{probunique} and \ref{probprice}:
\begin{theorem}
	\label{noexist}
Assume {\bf i,ii}. There exists ${\cal M}\in\cM(n, \vmu)$  for which there is a unique sub-partition. For any such ${\cal M}$ there exists $\vec{c}\in C(X,\R^n)$ satisfying {\bf iii}  for which there is no equilibrium price. 
\end{theorem} 
\begin{proof}
	Under condition (\ref{ass1}), if $B_1, \ldots B_n$ is a sub-partition corresponding to  ${\cal M}({\cal P})$, then \\
	$\mu(B_i\Delta A_i^0({\cal P}))=0$.
Indeed, 
	assume  $B_1, \ldots B_n$ is a sub-partition satisfying $\mu^j(B_i)=\mu^j(A^0_i({\cal P}))$. Then, by definition
	$$ 0=\int_X\left[\max_{i\in\{1\ldots n\}}\vp_i\cdot\vzeta(x)\right]_+\mu(dx)-
	\sum_{i=1}^n\vp_i\cdot\vm_i$$
	where $\vm_i:= (m_i^1({\cal P}), \ldots m_i^q({\cal P}))\in \R^q$. On the other hand
	$$\int_X\left[\max_{i\in\{1\ldots n\}}\vp_i\cdot\vzeta(x)\right]_+\mu(dx)\geq \sum_{i=1}^n \int_{B_i}\vp_i\cdot\zeta(x) \mu(dx)+ \int_{B_0}\left[\max_{i\in\{1\ldots n\}}\vp_i\cdot\vzeta(x)\right]_+\mu(dx)$$
	where $B_0:= X-\cup_i^n B_i$. Since $\mu^j(B_i)=m_i^j$ by assumption we get
	$$ \sum_{i=1}^n \int_{B_i}\vp_i\cdot\zeta(x) \mu(dx)= 	\sum_{i=1}^n\vp_i\cdot\vm_i \ . $$
	It follows that 
	$$\int_{B_0}\left[\max_{i\in\{1\ldots n\}}\vp_i\cdot\vzeta(x)\right]_+\mu(dx)\leq 0 $$
	and all inequalities above are, in fact, equalities.  In  particular, the last equality implies that $\vp_i\cdot\vzeta\leq 0$ $\mu-$a.e on $B_0$, so $B_0\cap(\cup_1^nA_i^0({\cal P}))=\emptyset$. 
	Moreover
	$$ \int_{B_i}\left(\left[\max_{k\in\{1\ldots n\}}\vp_k\cdot\vzeta(x)\right]_+-\vp_i\cdot\vzeta(x)\right)\mu(dx)=0$$
	so $\left[\max_{k\in\{1\ldots n\}}\vp_k\cdot\vzeta(x)\right]_+=\vp_i\cdot\vzeta(x)$ $\mu$-a.e on $B_i$ for $i>0$. This implies that $B_i$ is essentially contained in $A^0_i({\cal P})$
	i.e $\mu(B_i\ - A_i^0({\cal P}))=0$).
	Since  $\cup_{i=0}^n B_i=X$ by definition, we obtain the first  claim. 
	
  As for the second claim, 
 we introduce an example of a cost function $\vec{c}\in C(X, \R^n)$ satisfying {\bf iii}  for which an equilibrium price does not exist for ${\cal M}={\cal M}({\cal P})$ where ${\cal P}$ satisfies (\ref{ass1}). Since, by the first claim, $\vec{A}^0({\cal P})$ is  the unique sub-partition (up to $\mu$-negligible sets),  any equilibrium price ${\cal P}^{'}$ for ${\cal M}({\cal P})$ must satisfy
\be\label{A-=A-0} A_i^0({\cal P})=A_i({\cal P}^{'}) \ \ , i=1\ldots n\ee
up to $\mu-$negligible set, where $\vec{A}({\cal P}^{'})$ given by (\ref{Ap}) for ${\cal P}^{'}$ substituted for ${\cal P}$. By definition 
$$A_i^0({\cal P})\cap A_k^0({\cal P})
\subset \{x\in X; (\vp_i-\vp_k)\cdot\vzeta(x)=0\} \ . $$
Let such $i,k$ for which $A_i^0({\cal P})\cap A_k^0({\cal P})$ contains  $q+1$ distinct  points $x_1, \ldots x_{q+1}$. Let $V\subset \R^{q+1}$ 
be the subspace spanned by $\{(\vec{t}\cdot\vzeta(x_1), \ldots \vec{t}\cdot\vzeta(x_{p+1)})) \ ; \vec{t}\in \R^q\}$. Since $dim(V)\leq q$ it follows that there exists a vector $\vec{w}\in\R^{q+1}$, $\vec{w}\not\in V$. 
 We now choose the components $c_k, c_i$ of $\vec{c}$ such that $\left((c_i(x_1)-c_k(x_1)),\ldots (c_i(x_{q+1})-c_k(x_{q+1})\right)=\vec{w}$.
 
  It follows that for  the row vectors $\vec{p^{'}}_i$, $\vec{p^{'}}_k$ in ${\cal P}^{'}$ there exists $x_l$, $1\leq l\leq q+1$, 
   for which 
  $$ \vp^{'}_i\cdot\vzeta(x_l)+ c_i(x_l)\not= 
   \vp^{'}_k\cdot\vzeta(x_l)+ c_k(x_l)$$
   i.e. $x_l\not\in A_i({\cal P}^{'})\cap A_k({\cal P}^{'})$. This contradicts (\ref{A-=A-0}). 
   \end{proof}

\section{Generalization to multi-layer transport}
The definition of $\cM(n,\vmu)$ leads to a family of partial orders on the set of $\R^q$ valued measures on measure spaces. 
\begin{defi}
	Consider the set of $\R^q$ valued measure spaces ${\cal X}$. Given $n\in\mathbb{N}$, for $(X,\vmu), (Y,\vnu)\in{\cal X}$ define
	$$ (X,\vmu)\succ_n(Y,\vnu) \ \text{ iff}\ \ 
	\cM(n,\vmu) \supset \cM(n,\vnu) \ . $$
	Equivalently, for any $n-$sub-partition $\vec{B}$ of $Y$ there exists an $n-$sub-partition  $\vec{A}$ of $X$ such that
	$\nu^j(B_i)=\mu^j(A_i)$ for $i\in\{1\ldots n\}$ and $j\in\{1\ldots q\}$. 
	\end{defi}
It can easily be verified that $\succ_n$ is a partial order relation for each $n$. In particular, a necessary condition for  $(X,\vmu)\succ_n(Y,\vnu) $ 
is, by (\ref{nescond})
$\vmu(X)\geq \vnu(Y)$
coordinatewise (namely $\mu^j(X)\geq \nu^j(Y)$ for $j=1\ldots q$). 

In addition, this partial order  is preserved under weak* convergence:
$$ \lim_{k\rightarrow\infty}(\vmu_k,X)=(\vmu,X) \Longleftrightarrow \lim_{k\rightarrow\infty}\int\vpsi\cdot d\vmu_k
=\int\vpsi\cdot d\vmu$$
for any $\psi\in C(X; \R^q)$. 
\par\noindent  
{\bf Proposition} \cite{wolshlomi}: {\it 
	If $(\vmu_k,X)\succ_n(\vnu_k,Y)$ for any $k\in\mathbb{N}$ and 
	$\lim_{k\rightarrow\infty}(\vmu_k,X)=(\vmu,X), \lim_{k\rightarrow\infty}(\vnu_k,Y)=(\vnu,Y)$ then 
	$$ (X,\vmu)\succ_n(Y,\vnu) \ . $$
	}

Note also that $\succ_{n}\supset\succ_{n+1}$. As a result we can define the partial order  \be\label{prealln}\succ:=\cap_{n=1}^\infty\succ_n\ee
which is preserved under weak* convergence as well. This object is in some context  related to {\it convex} or {\it stochastic} order, as well as to {\it dominance}  and {\it majorization} of probability distributions (c.f. \cite{lamport94, bian1, bian2,joe1, joe2}). 
\par\noindent 
{\bf Theorem}\ \cite{wolshlomi}: \ {\it 
	The following are equivalent:\begin{itemize}
		\item $(\vmu,X)\succ(Y,\vnu)$.
		\item There exists a measurable kernel $P_x(dy)$ such that $\int_YP_x(dy)= 1$ $\mu-$as in $x$ and
		$$\int_X P_x(dy)\mu^j(dx)=\nu^j(dy) \ ,  \ j=1\ldots q \ . $$
		\item 
		For any non-negative convex function $f$ on $\R^q_+$
		$$ \int_X f\left(\frac{d\vmu}{d\mu}\right)d\mu\geq \int_Y f\left(\frac{d\vnu}{d\nu}\right)d\nu$$
		\item For any $k\in\mathbb{N}$ and any $k-$ sub-partition $B_1, \ldots B_k$ of $Y$ there is a $k-$sub-partition $A_1\ldots A_k$  of $X$, such that $\int_{A_i}d\mu^j=\int_{B_i}d\nu^j$ for $i=1\ldots k$ and $j=1\ldots q$. 
	\end{itemize}
}

The generalization of the Kantorovich problem for q-layer case is as follows:  Given a  pair of $\R^q-$ valued measures $(X,\vmu)\succ(Y,\vnu)$, minimize (\ref{Kanto}) on $\pi\in\Pi(\vmu,\vnu)$ where
$$ \Pi(\vmu,\vnu):=\left\{ \pi\in K(\mu,\nu)  \ , \ 
\int_X\frac{d\vmu}{d\mu}(dx)\pi(dxdy)= \vnu(dy) \  \right\} \ .  $$
The corresponding duality theorem is given in \cite{wolshlomi}:\par\noindent
{\bf Theorem} {\it 
	$$ \min_{\pi\in \Pi(\vmu, \vnu)}\int_{X\times Y} c(x,y)\pi(dxdy)=\sup_{(\vphi, \vpsi)\in J^q(c)}\int_X\vphi\cdot d\vmu+ \int_Y \vpsi\cdot d\vnu$$ where 
$$ J^q(c):= \left\{ \vphi\in C^q(X), \vpsi\in C^q(Y); \frac{d\vmu}{d\mu}\cdot \left(\vpsi(x)+\vphi(y)\right)\leq c(x,y)\ \ \forall (x,y)\in X\times Y  \right\} \ . $$
}
The manuscript \cite{wolshlomi} also extend the  the partial order (\ref{prealln}) and the  duality Theorem to Banach-valued measures.

In view of Theorem \ref{noexist}, the existence of a maximizer to the dual problem is not guaranteed, in general. The following question is still open:
\par\noindent
{\bf Question}:\ {\it  What is the conditions on $c=c(x,y)$, $(\vmu,X)$, $(\vnu,Y)$ under which a maximizer to the dual problem exists in the Banach (or even the finite dimensional vector) valued  case?}

 There is a  relevant result in this direction:
\par\noindent 
{\bf Theorem} \  \cite{wolar}: {\it 
 Consider  the semi-discrete case. If ${\cal M}$ is a relative interior point of $\cM(n,\vmu)$ then there is a maximizer of the dual problem in the finite dimensional vector valued case:
$$ \sup_{\cal P} Tr({\cal P}{\cal M}) - \int_X\phi(x;{\cal P})d\mu \ .  $$ 
If assumption {\bf i-iii} are satisfied than the maximum is unique. }


\begin{thebibliography}{9}
		\bibitem{Lyapel}
	David A. Ross: \emph{An Elementary Proof of Lyapunov's Theorem}, The American Mathematical Monthly
	Vol. 112, No. 7, 651-653, 2005
	\bibitem{wolshlomi} Gover, S.: \emph{Duality theorems for optimal transport problems}, Ph.D thesis, in preparation

	\bibitem{Partelli}
	 Pratelli, A.:\emph{On the equality between Monge's infimum and Kantorovich?s minimum in optimal mass transportation},  Ann. Inst. H. Poincar´e (B) Probab. Statist., 43, 1-13 , 2007
	\bibitem{BK} Bogachev, V.I., Kalinin, A.N. \& Popova, S.N. \emph{On the Equality of Values in the Monge and Kantorovich Problems},  J Math Sci 238, 377-389, 2019
	
	
	\bibitem{bian1}
	Bianchini S. ,
	\emph{The Vector Measures Whose Range Is Strictly Convex},
	Journal of Mathematical Analysis and Applications 232, 1-19 , 1999 
	
	
	\bibitem{bian2}
	Bianchini S, Cerf R, and Mariconda C.
	\emph{Chebyshev measures and the vector measures whose range is strictly convex.}
	Atti Sem. Mat. Fis. Univ. Modena 46 , no. 2, 525-534, 1998
	\bibitem{lamport94}
	Blackwell D.,
	\emph{Comparison of Experiments},
	Proc. Second Berkeley Symp. on Math. Statist. and Prob. (Univ. of Calif. Press), 93-102,  1951
	
	
%
%
	
	
	
	
	
	
	\bibitem{Gal}
	Galichon, A: \emph{Optimal Transport Methods in Economics},
	Princeton University Press, 2016
%
%
	
	
	\bibitem{Ka} L. Kantorovich,  \emph{ On the translocation of masses.},  C.R. (Doklady) Acad. Sci. URSS (N.S.), 37:199-
	201, 1942.
	\bibitem{LW}Legut, J and Wilczy\'{n}ski, M,
	\emph{How to obtain a range of a nonatomic vector measure in $\R^2$},
	, J. Math. Anal. Appl. 394, 102-111, 2012
	\bibitem{Lyap}  Lyapunov, A. \emph{Sur les fonctions-vecteurs completement additives.}
	Bull. Acad. Sci. URSS 6, 465-478, 
	1940
	
	
	
	
	
	
	
	
	
	
	\bibitem{joe1}
	Joe, H.
	\emph{Majorization and divergence},
	J. Math. Anal. Appl. 148 , 1990
	
	\bibitem{joe2}
	Joe, H.
	\emph{Majorization, randomness and dependence for multivariate distributions},
	Ann. Probab. 15, no. 3, 1217-1225, 1987
	
	
	\bibitem{Mo}Monge, G.:  \emph{ M\'{e}moire sur la th\'{e}orie des d\'{e}blais et des remblais}, In  Histoire de l\'{A}cad\'{e}mie Royale des Sciences de Paris,  666-704, 1781
	
	
	
	
	
	\bibitem{SntA}Santambrogio, F.: \emph{Optimal Transport for Applied Mathematicians: Calculus of Variations, PDEs, and Modeling},
	Volume 87 of Progress in Nonlinear Differential Equations and Their Applications
	Publisher	Birkhuser, 2015
	
	
	\bibitem{Rec1} S.T. Rachev and L.  R\"{u}schendorf : \emph{Mass Transportation Problems: Volume I: Theory. Vol. 1}, Springer, 1998
	
	\bibitem{Vil1}   Villani, C.:  \emph{Topics in Optimal Transportation}, vol. 58 of Graduate Studies in Mathematics, AMS, Providence, RI, 2003
	\bibitem{Vil2} Villani,  C.:\emph{Optimal Transport, old and new}, Springer 2009
		\bibitem{wolar} Wolansky, G.: 
	\emph{Optimal sub-partitions and Semi-discrete optimal transport}, 
	arXiv:1911.04348
	
	\bibitem{wol}
	Wolansky, G.
	\emph{ On semi-discrete Monge-Kantorovich and generalized sub-partitions},
	J. Optim. Theory Appl. 165, no. 2, 3592384, 2015
	
	\bibitem{wol1} Wolansky, G.:
	\emph{On optimal sub-partitions, individual values and cooperative games: does a wiser
		agent always produce a higher value? }
	Math. Financ. Econ. 11 , no. 1, 85?109, 2017
	
		\bibitem{wolar} Wolansky, G.: 
	\emph{Optimal sub-partitions and Semi-discrete optimal transport}, 
	arXiv:1911.04348
	
	
	
	
	
	
	
	
	
	
	
	
	
	
	
	
	
	
	
	
	
	
	
	
	
	
	
	
	
	
	
	
	
\end{thebibliography}
\end{document}